\documentclass[a4paper,twoside,reqno]{amsart}

\usepackage[pagewise]{lineno}
\usepackage[utf8]{inputenc}

\usepackage{authblk}
\usepackage{hyperref}
\hypersetup{urlcolor=blue, citecolor=red}
\usepackage{amsmath,amsthm,amssymb,xcolor,bbm,mathrsfs}
\usepackage[english]{babel}
\usepackage{fancyhdr}
\newcommand{\indicator}[1]{\mathbbm{1}\{#1\}}

\DeclareMathOperator{\card}{card}

\makeatletter
\@namedef{subjclassname@1991}{2020 Mathematics Subject Classification}
\makeatother

\title{An imitation model based on the majority}
\author{Hsin-Lun Li}

\date{}
\email{hsinlunl@asu.edu}

\theoremstyle{definition}
\newtheorem{theorem}{Theorem}

\newtheorem{lemma}[theorem]{Lemma}
\newtheorem{corollary}[theorem]{Corollary}

\begin{document}

\allowdisplaybreaks

\thispagestyle{firstpage}
\maketitle
\begin{center}
    Hsin-Lun Li
    \centerline{$^1$National Sun Yat-sen University, Kaohsiung 804, Taiwan}
\end{center}
\medskip

\begin{abstract}
The voter model consists of a set of agents whose opinion is a binary variable. At each time step, an agent along with a social neighbor is selected and the agent imitates the social neighbor at the next time step. In this paper, we study a variant of the voter model: an imitation model based on the majority. The agent imitates the selected social neighbor if and only if the number of neighbors following the opinion of the selected social neighbor is more than that following the agent. We study circumstances under which a consensus can not be achieved and the probability of consensus on a finite connected social graph.
\end{abstract}

\section{Introduction}
The voter model consists of a set of agents whose initial opinion is a binary variable $s=\pm 1$. At each time step, an agent $i$ along with one of its neighbors $j$ is selected and agent $i$ imitates agent $j$ at the next time step. The voter model differs from the Deffuant model and the Hegselmann-Krause model studied in~\cite{mHK,lanchier2020probability,MR3069370,lanchier2022consensus,mHK2}. It has no confidence threshold in opinions. In this paper, we consider a variant of the voter model: an imitation model based on the majority. The mechanism is as follows. Given a finite set of agents, say $[n]=\{1,2,\ldots,n\}$. Let $x_i(t)$ be the opinion of agent $i$ at time $t$ and $x_i(0), i\in [n]$ are binary random variables with state space $$S=\{\pm 1\}\quad \hbox{and} \quad P(x_i(0)=1)=p.$$ At each time step, an agent $i$ along with one of its neighbors $j$ is selected and agent $i$ imitates agent $j$ at the next time step if and only if the number of social neighbors of agent $i$ following the opinion of agent $j$ is more than that following the opinion of agent $i$. Interpreting in a graph, each vertex stands for an agent and an edge connecting two vertices indicates the corresponding agents are social neighbors. Let $G=([n], E)$ be a connected social graph with vertex set and edge set $[n]$ and $E$, let $N_i=\{k\in [n]:k\neq i \ \hbox{and}\ (i,k)\in E\}$ be the collection of all social neighbors of agent $i$ excluding itself, and let $k_t$ be the vertex selected to imitate the selected social neighbor or not at time~$t$, which is a random variable with state space $[n]$. The selected social neighbor of vertex $k_t$ is of state space $N_{k_t}$. Interpreting the mechanism in math,
\begin{align*}
    & x_{k_t}(t)\neq x_{k_t}(t+1)\ \hbox{if and only if}\\
    &\hspace{1cm}\card(\{j\in N_{k_t}:x_j(t)=x_{k_t}(t+1)\})>\card(\{j\in N_{k_t}:x_j(t)=x_{k_t}(t)\})
\end{align*}
for the cardinality of a set $S$ denoted by $\card(S)$.

\section{Main results}
Unlike the voter model on a finite connected social graph, the imitation model based on the majority does not always achieve a consensus. Let $\mathscr{C}$ be the collection of all sample points leading to a consensus, i.e.,
$$\mathscr{C}=\{\max_{i,j\in [n]}(x_i(T)-x_j(T))=0\}.$$
\begin{theorem}\label{thm:consensus}
We have $P(\mathscr{C})=1$ if $G$ is complete and $P(\mathscr{C})\geq 1-2p(1-p)|E|$ for all social graphs $G$. 
\end{theorem}

A tree is a connected graph with the least edges. The social graph $G$ is assumed connected. Therefore, $G$ is a tree if and only if the number of edges in $G$ is $n-1$. 
\begin{corollary}
We have  $P(\mathscr{C})\geq 1/2|E|$ if $p\leq 1/2|E|$. In particular, $P(\mathscr{C})\geq \frac{1}{2(n-1)}$ if $p\leq \frac{1}{2(n-1)}$ and $G$ is a tree.
\end{corollary}

\section{The model}
\begin{lemma}\label{key}
Let $ Z_t=\sum_{i,j\in [n]}\mathbbm{1}\{x_i(t)\neq x_j(t)\}\vee \mathbbm{1}\{(i,j)\notin E\}$. Then, $Z_t$ is nonincreasing with respect to $t$. In particular,
\begin{align*}
   &Z_t-Z_{t+1}=2\big(\card(\{i\in N_{k_t}: \ x_i(t)=x_{k_t}(t+1)\})\\
   &\hspace{6cm}-\card(\{i\in N_{k_t}: x_i(t)=x_{k_t}(t)\})\big).
\end{align*}
\end{lemma}

\begin{proof}
For all $t\geq 0$, let $k=k_t$, $x_i=x_i(t)$ and $x_i^\star=x(t+1)$ for all $i\in [n]$. Then, 
\begin{align*}
    &Z_t-Z_{t+1}=\sum_{i,j\in [n]}\big(\indicator{x_i\neq x_j}\vee\indicator{(i,j)\notin E}-\indicator{x_i^\star\neq x_j^\star}\vee\indicator{(i,j)\notin E}\big)\\
    &\hspace{0.5cm}=2\sum_{i\in [n]-\{k\}}\big(\indicator{x_i\neq x_k}\vee\indicator{(i,k)\notin E}-\indicator{x_i\neq x_k^\star}\vee\indicator{(i,k)\notin E}\big)\\
    &\hspace{0.5cm}=2\sum_{i\in N_k}\big(\indicator{x_i\neq x_k}-\indicator{x_i\neq x_k^\star}\big)\\
    &\hspace{0.5cm}=2\big(\card(\{i\in N_k: \ x_i=x_k^\star\})-\card(\{i\in N_k: x_i=x_k\})\big).
\end{align*}
\end{proof}

\begin{lemma}[Uniform convergence]\label{uniform convergence}
    We have $x_i(t)=x_i(t+1)$ for all $i\in [n]$ after some finite time almost surely.
\end{lemma}

\begin{proof}
$(Z_t)_{t\geq 0}$ is a supermartingale uniformly bounded in $L^1$, i.e., $E|Z_t|\leq E|Z_0|<\infty$ for all $t\geq 0$. By the martingale convergence theorem, $Z_t$ converges almost surely to $Z_\infty\in L^1$. It follows that $$\lim_{t\to\infty}(Z_t-Z_{t+1})=Z_\infty-Z_\infty=0.$$ Since $Z_t-Z_{t+1}\in\mathbf{N}$, $Z_t=Z_{t+1}$ after some finite time. Therefore, $x_i(t)=x_i(t+1)$ for all $i\in [n]$ after some finite time.  

\end{proof}
Let $T$ be the earliest time that no agents imitate the other social neighbors, i.e., $$T=\inf\{t\geq 0: x_i(s)=x_i(s+1)\ \hbox{for all}\ i\in [n]\ \hbox{and}\ s\geq t \}.$$ It follows from Lemma~\ref{uniform convergence} that $T$ is almost surely finite. Via the second Borel-Cantelli lemma, a vertex along with each social neighbor is selected infinitely many times, therefore engendering the following lemma.

\begin{lemma}\label{lemma:fact}
    We have $$\card(\{j\in N_{i}:x_j(T)=-x_{i}(T)\})\leq\card(\{j\in N_{i}:x_j(T)=x_{i}(T)\})$$
    for all $i\in [n]$.
\end{lemma}

\begin{lemma}\label{equivalent statements for consensus}
    The following statements are equivalent:
    \begin{enumerate}
        \item a consensus can be achieved,\label{consensus}\vspace{2pt}
        \item $x_i(T)=x_j(T)$ for all $i\in [n]$ and $j\in N_i$, and\label{neighborhood}\vspace{2pt}
        \item $Z_T=n(n-1)-2|E|$.\label{identity for consensus}
    \end{enumerate}
\end{lemma}

\begin{proof}
   It is clear that \ref{consensus} implies \ref{neighborhood}. For \ref{neighborhood}$\implies$\ref{identity for consensus}, observe that
   \begin{align}
       Z_t&=\sum_{i=1}^n\big(\sum_{j\in N_i}\indicator{x_i(t)\neq x_j(t)}+\sum_{j\in N_i^c-\{i\} }1\big)\notag\\
       &=\sum_{i=1}^n\sum_{j\in N_i}\indicator{x_i(t)\neq x_j(t)}+\sum_{i=1}^n(n-\deg_G(i)-1)\notag\\
       &=\sum_{i=1}^n\sum_{j\in N_i}\indicator{x_i(t)\neq x_j(t)}+n(n-1)-2|E|\label{handshaking lemma}
   \end{align}
   where $\deg_G(i)$ is the degree of vertex $i$ in $G$ and \eqref{handshaking lemma} follows the handshaking lemma. Hence, $Z_T=n(n-1)-2|E|.$ For \ref{identity for consensus}$\implies$\ref{consensus}, \eqref{handshaking lemma} implies $$\sum_{i=1}^n\sum_{j\in N_i}\indicator{x_i(T)\neq x_j(T)}=0,\ \hbox{therefore}\ x_i(T)=x_j(T)$$  for all $i\in [n]$ and $j\in N_i$. Due to the connectedness of $G$, there is an $i,\ j$-path $P$ in $G$ for all vertices $i$ and $j$. Since any two adjacent vertices on $P$ are social neighbors having the same opinion at time $T$, $x_i(T)=x_j(T)$ for all $i,\ j\in [n].$ Hence, a consensus can be achieved.
\end{proof}

\begin{lemma}\label{lemma:complete}
    $G$ complete implies a consensus achieved eventually almost surely. 
\end{lemma}

\begin{proof}
    Assume that this is not the case. By Lemma~\ref{equivalent statements for consensus}, there are $i\in [n]$, $j\in N_i$ and $x_i(T)\neq x_j(T)$, say $x_i(T)=1$ and $x_j(T)=-1$. Also via Lemma~\ref{lemma:fact}, $$\card(\{k\in [n]-\{i\}:x_k(T)=1\})\geq \card(\{k\in [n]-\{i\}:x_k(T)=-1\}),$$
    therefore
    $$\card(\{k\in [n]-\{i\}:x_k(T)=1\})\geq \frac{n-1}{2}.$$
    So 
    \begin{align*}
       &\card(\{k\in [n]-\{j\}:x_k(T)=1\})\geq \frac{n-1}{2}+1\\
       &\hspace{3cm}>\card(\{k\in [n]-\{j\}:x_k(T)=-1\}),\ \hbox{a contradiction.} 
    \end{align*} 
\end{proof}

\begin{lemma}\label{nonconsensus}
    Let $G$ be a cycle or a path with $n\geq 4$. Then, a consensus can be achieved if and only if no four consecutive vertices correspond to opinions $1,\ 1,\ -1,\ -1$ at the initial time.
\end{lemma}

\begin{proof}
    $(\implies)$\ Assume that this is not the case. Without loss of generality, consider the cycle $1\rightarrow 2\rightarrow \ldots\rightarrow n\rightarrow 1$ with $x_1(0)=1$, $x_2(0)=1$, $x_3(0)=-1$ and $x_4(0)=-1$. We claim by induction that vertices 1, 2, 3 and 4 remain unchanged in opinion. It is true for $t=1$ regardless of what vertex along with one of its neighbors is selected. If this is true for $t$, then this is true for $t+1$. Therefore, $x_1(t)=1$, $x_2(t)=1$, $x_3(t)=-1$ and $x_4(t)=-1$ for all $t\geq 0.$ This is also true if $G$ is a path. So a consensus can not be achieved, a contradiction. 

    $(\impliedby)$\ Assume that this is not the case. Following Lemma~\ref{equivalent statements for consensus}, consider the cycle $1\rightarrow 2\rightarrow \ldots\rightarrow n\rightarrow 1$ with $x_1(T)=1$ and $x_2(T)=-1$. Then, $x_3(T)=-1$ and $x_n(T)=-1$, contradicting $x_1(T)=1$. Considering the path by removing edge $(1,n)$ from $G$, still contradict $x_1(T)=1.$
    
\end{proof}

Lemma~\ref{nonconsensus} gives rise to the following lemma.

\begin{lemma}
    A consensus can not be achieved for all social graphs $G$ having a path $a\rightarrow b\rightarrow c\rightarrow d$ corresponding to opinions $1,\ 1,\ -1,\ -1$ at the initial time and vertices $a$, $b$, $c$ and $d$ of degree at most 2 in $G$.
\end{lemma}

\begin{proof}[\bf Proof of Theorem~\ref{thm:consensus}]
It follows from Lemma~\ref{lemma:complete} that $P(\mathscr{C})=1$ if $G$ is complete.
    Let $q=1-p$. Following Lemmas~\ref{key} and \ref{uniform convergence}, $(Z_t)_{t\geq 0}$ is a bounded supermartingale and $T$ is almost surely finite. By the optional stopping theorem, 
    $$ E(Z_T)\leq E(Z_0),$$ 
    therefore via \eqref{handshaking lemma} and conditioning on $\mathscr{C}^c$ the complement of $\mathscr{C}$,
    $$E\big(\sum_{i=1}^n\sum_{j\in N_i}\indicator{x_i(T)\neq x_j(T)}|\mathscr{C}^c\big)P(\mathscr{C}^c)\leq \sum_{i=1}^n\sum_{j\in N_i}E(\indicator{x_i(0)\neq x_j(0)}).$$ 
    By Lemma~\ref{equivalent statements for consensus}, $\indicator{x_i(0)\neq x_j(0)}=\hbox{Bernoulli}(2pq)$ and handshaking lemma,
    $$2(1-P(\mathscr{C}))\leq 4pq|E|,$$
    therefore
    $$P(\mathscr{C})\geq 1-2pq|E|.$$
    
\end{proof}

\section*{Acknowledgment}
The author is funded by the National Science and Technology Council in Taiwan.

\end{document}